\documentclass{article}

\usepackage{arxiv}

\usepackage{cite}
\usepackage[utf8]{inputenc}
\usepackage[T1]{fontenc}
\usepackage{hyperref}
\usepackage{url}
\usepackage{booktabs}
\usepackage{amsfonts}
\usepackage{nicefrac}
\usepackage{microtype}
\usepackage{amsthm}
\usepackage{amsmath}
\usepackage{lipsum}
\usepackage{graphicx}
\graphicspath{ {./images/} }

\usepackage{graphicx, color}
\usepackage{hyperref}

\newtheorem{theorem}{Theorem}[section]

\theoremstyle{definition}
\newtheorem{definition}[theorem]{Definition}

\usepackage{float}
\theoremstyle{remark}

\makeatletter
\def\th@plain{
	\thm@notefont{}
	\itshape
}
\def\th@definition{
	\thm@notefont{}
	\normalfont
}
\def\th@remark{
	\thm@notefont{}
	\normalfont
}
\makeatother

\title{Dynamics of the Modified Chebyshev's method to multiple roots}

\author{
	{\large Diego Linares}\textsuperscript{1}
	\and
	{\large \textbf{Carlos Cadenas}}\textsuperscript{1,2}
}

\begin{document}
	\maketitle
	
	\begin{center}
		\footnotesize
		\textsuperscript{1}Departamento de Matem\'aticas, Facultad Experimental de Ciencias y Tecnolog\'{\i}a, Universidad de Carabobo, Venezuela\\
		\textsuperscript{2}Centro Multidisciplinario de Visualizaci\'on y C\'omputo Cient\'{\i}fico (CEMVICC), Universidad de Carabobo, Venezuela\\[2ex]
		\textit{Corresponding author: \texttt{ccadenas@uc.edu.ve}}
	\end{center}
	\begin{abstract}
		This study explores the complex dynamics of the rational function associated with the Modified Chebyshev's root-finding method. After introducing the basic preliminaries of discrete dynamical systems, we analyze the dynamical behavior of the method, classifying the stability of its fixed points and critical orbits. These theoretical findings are then illustrated through dynamical planes, which map the basins of attraction and reveal the convergence characteristics and potential chaotic regions of the method.
	\end{abstract}
%%%%% end %%%%%%%%%%%

%%%%% Keywords %%%%%%%%%%%

%%%% AMS subject classifications %%%%
%\ams{65H05\newline}

	\section{Introduction}

	Specialized algorithms are required for solving nonlinear equations with multiple roots. Many of these have been developed based on Schröder's fundamental ideas (see, for example, \cite{SAkram2019, SAkram2021, HArora2022}, \cite{CECadenas2018a, CECadenas2018b}), including the modified versions of the Newton, Halley, and Chebyshev methods.

	A well-established approach to comparing such iterative methods is through the analysis of their complex dynamics. This is typically addressed from two perspectives: the study of basins of attraction for fixed points and the use of conjugation mappings to explore the parameter space. Regarding the former, several studies focusing on multiple roots can be found in the literature \cite{KKneisl2001, MKansal2020, DKumar2020, FChicharro2020, FChicharro2022, DKumar2018}. Similarly, the latter approach has been extensively documented in works such as \cite{ESchroder1870, Cadenas2024, BNeta2022, Torregrosa2022, Vazquez2018, FZafar2018a, FZafar2018, FZafar2018b, FZafar2022,  TSingh2023, FZafar2023}.

	In \cite{CECadenas2019}, a methodology was introduced to study the dynamics of the modified Newton’s method by utilizing a parameter that establishes the relationship between the multiplicities of roots in polynomials with two distinct roots. This framework was recently applied in \cite{CECadenas2023a} to analyze the modified Halley method. In the present paper, we extend this approach to investigate the dynamics of the modified Chebyshev method.

	The remainder of this paper is organized as follows. Section 2 presents the basic preliminaries and the theoretical framework necessary for the study. In Section 3, we analyze the dynamical behavior of the rational function associated with the modified Chebyshev method, the dynamical planes, and the basins of attraction to visualize the stability of the method. Finally, the results are presented and discussed in Section 4.

	\section{Basic preliminaries}

	In this section, we present the fundamental definitions necessary for the subsequent dynamical analysis. While these concepts are standard in the field of complex dynamics and can be found in classic references such as \cite{AFBeardon1991} and \cite{Urbanski2023}, they are included here to ensure the paper remains self-contained.\newpage

	\begin{definition}
		The \emph{Riemann sphere}, also known as the \emph{extended complex plane}, is defined as
		\[
		\widehat{\mathbb{C}} = \mathbb{C} \cup \{\infty\}.
		\]
	\end{definition}

\begin{definition}
	Let \( R : \widehat{\mathbb{C}} \to \widehat{\mathbb{C}} \) be a rational function and \( \alpha \in \widehat{\mathbb{C}} \). The \( k \)-th iterate of \( R \) at \( \alpha \) is defined as
	\[
	R^k(\alpha) = \underbrace{R \circ R \circ \cdots \circ R}_{k \text{ times}}(\alpha).
	\]
\end{definition}

\begin{definition}
	For \( z \in \widehat{\mathbb{C}} \), the \emph{orbit} of \( z \) under the action of \( R \) is the set
	\[
	\operatorname{Orb}(z) = \{z, R(z), R^2(z), \cdots, R^n(z), \cdots\}.
	\]
\end{definition}

\begin{definition}
	A point \( z_0 \in \widehat{\mathbb{C}} \) is called a \emph{fixed point} of \( R \) if \( R(z_0) = z_0 \).
\end{definition}

\begin{definition}
	A point \( z_0 \) is called \emph{periodic} of (prime) period \( m > 1 \) if \( R^m(z_0) = z_0 \) and \( R^k(z_0) \neq z_0 \) for all \( 0 < k < m \).
\end{definition}

\begin{definition}
	A point \( z_0 \) is called \emph{pre-periodic} if it is not periodic, but there exists an integer \( k > 0 \) such that \( R^k(z_0) \) is periodic.
\end{definition}

\begin{definition}
	A point \( z_{\mathrm{cr}} \in \widehat{\mathbb{C}} \) is called a \emph{critical point} of \( R \) if \( R'(z_{\mathrm{cr}}) = 0 \).
\end{definition}

\begin{definition}
Let \( S : \bar{\mathbb{C}} \to \bar{\mathbb{C}} \) be a rational function with \(\deg(S) \geq 2\). A fixed point \( z_0 \) of \( S \) is classified as follows:
\[
\begin{cases}
	|S'(z_{0})| = 0 & \Rightarrow z_{0} \text{ is superattracting}, \\[4pt]
	|S'(z_{0})| < 1 & \Rightarrow z_{0} \text{ is attracting}, \\[4pt]
	|S'(z_{0})| > 1 & \Rightarrow z_{0} \text{ is repelling}, \\[4pt]
	|S'(z_{0})| = 1 & \Rightarrow z_{0} \text{ is neutral}. \\[4pt]
\end{cases}
\]
\end{definition}

\begin{definition}
	A fixed point \( z_0 \) of \( R \) which does not correspond to a root of the function is called a \emph{strange fixed point}.
\end{definition}

\begin{definition}
	Consider a rational function \( R: \hat{\mathbb{C}} \to \hat{\mathbb{C}} \).
	The \emph{Fatou set} \( \mathcal{F}(R) \) is the maximal open subset of the Riemann sphere where the family of iterates \( \{ R^n \}_{n \in \mathbb{N}} \) is equicontinuous.
	The \emph{Julia set} \( \mathcal{J}(R) \) is defined as its complement: \( \mathcal{J}(R) = \hat{\mathbb{C}} \setminus \mathcal{F}(R) \).
\end{definition}

\begin{definition}
	For a fixed point \( \alpha \) of a rational function \( R \), its \emph{basin of attraction} is the set of initial points whose orbits converge to \( \alpha \):
	\[
	\mathcal{A}(\alpha) = \{ z \in \hat{\mathbb{C}} \; : \; \lim_{n \to \infty} R^n(z) = \alpha \}.
	\]
\end{definition}

\begin{definition}\label{def:conjugacy}
	A \emph{Möbius transformation} is a rational function of the form
	\[
	M(z) = \frac{az + b}{cz + d}, \quad \text{with } ad - bc \neq 0.
	\]
	Two rational functions \( R_1 \) and \( R_2 \) are \emph{conjugate} if there exists a Möbius transformation \( M \) satisfying
	\[
	R_2 = M \circ R_1 \circ M^{-1}.
	\]
\end{definition}

\section{Dynamical behavior of the rational function associated with the method in study}

This section is devoted to the study of the dynamical behavior associated with the Modified Chebyshev's method for solving nonlinear equations. The primary objective is to analyze the rational function derived from applying this iterative scheme to a generic polynomial, focusing on its stability and convergence properties through the tools of complex dynamics.

We begin by introducing the iterative expression of the method under study. For a function \( f \), and given an initial estimate \( z_0 \), the Modified Chebyshev's method is defined by the recurrence relation:
\begin{equation}\label{ModifiedChebychev}
	z_{n+1} = z_n - \frac{m \, f(z_n)}{2 f'(z_n)} \left( 3 - m + m \, L_f(z_n) \right), \quad \text{with} \quad L_f(z) = \frac{f(z)f''(z)}{[f'(z)]^2},
\end{equation}
where \( m \) denotes the multiplicity of the root to be approximated and $L_f(z)$ is called the logarithmic convexity function \cite{MAHernandez1992}.

The application of this method to a generic polynomial yields a rational function, \( R(z) \), whose iterative behavior governs the convergence of the algorithm. The study of the dynamics of \( R(z) \) allows us to understand the global behavior of the method, identifying regions of stability associated with the roots of the polynomial, as well as chaotic regions and other unstable behaviors that can hinder convergence.

\subsection{Conjugacy classes}

Here we establish the conjugacy class  and the analytical expressions for the fixed and critical points  of this family in terms of the parameter $K$. In this work, we investigate the dynamics of the rational operator $R_f(z)$ associated with the modified Chebyshev method (\ref{ModifiedChebychev}), defined as:

\begin{equation}\label{familyR}
	R_f(z) = z - \frac{mf(z)}{2f'(z)}\left(3 - m + mL_f(z)\right).
\end{equation}

To conduct this analysis, we apply the operator to the specific family of polynomials $p(z) = (z-a)^m(z-b)^n$, where $a \neq b$ and $m = Kn$ (with $K$ representing the ratio of multiplicities). A fundamental tool for simplifying this study is the concept of analytic conjugacy, which allows us to generalize the behavior of the method across entire classes of functions. In this sense, we recall the definition \ref{def:conjugacy}:

The operator (\ref{familyR}) satisfies the Scaling Theorem \ref{t:scaling}, a vital property that ensures the dynamical study of a particular polynomial can be extended to any other polynomial under an affine transformation:

\begin{theorem}\label{t:scaling} (The Scaling Theorem).
	Let $f(z)$ be an analytic function on the Riemann sphere, and let $T(z)=\alpha z+\beta$ with $\alpha\neq0$ be an affine map. If $g(z)=(f \circ T)(z)$, then $R_g = T^{-1} \circ R_f \circ T$. That is, the rational map $R_f$ is analytically conjugate to $R_g$ through $T$.
\end{theorem}

\begin{proof}
	With the iteration function $R(z)$, we have

	$$R_g(T^{-1}(z)) = T^{-1}(z)-  \frac{mg\left(T^{-1}(z)\right)}{2g'\left(T^{-1}(z)\right)}\left(3-m+m L_{g}\left(T^{-1}(z)\right)\right) \nonumber$$

	Since
	$\alpha  T^{-1}(z)+\beta=z$,  $g\circ T^{-1}(z)=f(z)$ and $\left(g\circ T^{-1}\right)'(z)=\frac{1}{\alpha}g'\left(T^{-1}(z)\right)$, then

	 $$g'\left(T^{-1}(z)\right)=\alpha\left(g\circ T^{-1}\right)'(z)=\alpha f'(z).$$

	This implies

	\begin{eqnarray*}
		T\circ R_g\circ T^{-1}(z) &=& T\left(R_g(T^{-1}(z))\right)=\alpha R_g\left(T^{-1}(z)\right)+\beta \\
		&=&\alpha   T^{-1}(z)-  \frac{m \alpha g\left(T^{-1}(z)\right)}{2g'\left(T^{-1}(z)\right)}\left(3-m+m L_{g}\left(T^{-1}(z)\right)\right) +\beta\\
		&=&z-\frac{mf(z)}{2f'(z)}\left(3-m+mL_f(z)\right))\\
		&=&R_f(z)
	\end{eqnarray*}

\end{proof}

The Scaling Theorem stated above implies that, through an appropriate change of coordinates, the dynamical analysis of the iteration function given in (\ref{ModifiedChebychev}) for general polynomials can be reduced to the study of its dynamics applied to simpler polynomial forms. The following theorem \ref{t:mio} presents a universal Julia set for quadratic polynomials under the method (\ref{ModifiedChebychev}).

\begin{theorem}\label{t:mio} For a rational map  $R_p(z)$ arising from the method (\ref{ModifiedChebychev}) applied to $p(z)=(z-a)^m(z-b)^n, a\neq b$, where $m=Kn$,  $R_p(z)$  is conjugate via the M\"{o}bius transformation given by $M(z)=\frac{z-a}{z-b}$ to

	\begin{equation}\label{S}
		S(z)=\frac{z^3[2z+K(K+3)]}{(K-1)(K-2)z^3+6K(1-K)z^2+2K^2(3-K)z+2K^3}
	\end{equation}

\end{theorem}

\begin{proof}
	Let $p(z)=(z-a)^m(z-b)^n, a\neq b$  with  $m=Kn$ and let $M$ be the M\"{o}bius transformation  given by $M(z)=\frac{z-a}{z-b}$ with its inverse  given by $M^{-1}(u)=\frac{bu-a}{u-1}$, which may be considered as a map from $\mathbb{C}\cup\{\infty\}$. We then have

	\begin{eqnarray*}
		M\circ R_p\circ M^{-1}(u)&=&M\circ R_p\left(\frac{bu-a}{u-1}\right)\nonumber\\
		&=&\frac{u^3[2u+K(K+3)]}{(K-1)(K-2)u^3+6K(1-K)u^2+2K^2(3-K)u+2K^3}
	\end{eqnarray*}
\end{proof}

The parameters \( a \) and \( b \) are absent from the conjugate map (\ref{S}), as they cancel out through the Möbius transformation, consistent with the Scaling Theorem verified by this method.

The dynamical analysis is then performed by examining specific values of the parameter \( K \) and studying the strange fixed points that depend on this parameter. The basins of attraction corresponding to these fixed points are discussed, which were produced by modifying the code provided in \cite{CECadenas2017d}.

\subsection{Study of the fixed points}
When analyzing the dynamics of a rational map on the Riemann sphere, a fixed point is defined as a point that remains invariant under iteration, satisfying the condition $S(z) = z$. In this work, we determine the fixed points of the conjugacy map (\ref{S}) by solving this equation.

Fixed points are classified into two categories: \emph{regular fixed points}, which correspond to the roots of the original polynomial and typically serve as attractors, and \emph{strange fixed points}, which arise from the rational form of the iterative scheme and are not associated with any roots of the polynomial. For numerical stability, it is desirable for these strange fixed points to be repelling rather than attracting.

Solving the equation $S(z) = z$ yields the following fixed points:

$
\begin{aligned}
	z_{0} &= 0, \\[4pt]
	z_{1} &= 1, \\[4pt]
	z_{2, 3} &= \frac{(3 \pm \sqrt{2K + 3})K}{K - 3}, \\[4pt]
	z_{4} &= \infty.
\end{aligned}
$

\subsubsection{Analysis of the case $K=-2$}
In this case, the conjugacy map is given by:
\begin{equation}\label{S-2}
	S(z)=\frac{z^3}{2(3z^2-6z+4)}
\end{equation}
The derivative of (\ref{S-2}) is:
\begin{equation}\label{S'-2}
	S'(z)=\frac{3z^2(z-2)^2}{2\left(3z^2-6z+4\right)^2}
\end{equation}
From (\ref{S'-2}), the strange fixed points are:
$z=\frac{2(3\pm i)}{5}$, where $\left|S'\left(\frac{2(3\pm i)}{5}\right)\right|=3 > 1$. Thus, both are repelling fixed points.

On the other hand, the fixed points $z=0$ and $z=\infty$ are superattracting. The point $z=2$ is a critical point of multiplicity two; if we choose $z_0=2$ as an initial guess, the iteration $z_{n+1}=S(z_n)$ using $S$ in (\ref{S-2}) converges to zero.

\subsubsection{Analysis of the case $K=1$}
In this case, the conjugacy map is given by:
\begin{equation}\label{S1}
	S(z) = \frac{z^3(z+2)}{2z+1}
\end{equation}
The derivative is calculated as:
\begin{equation}\label{Sprime1}
	S'(z) = \frac{6z^2(z+1)^2}{(2z+1)^2}
\end{equation}
From (\ref{Sprime1}), the strange fixed points are determined to be:
$z=1$ with $|S'(1)| = 1.5 > 1$, which is a repelling fixed point (repulsor). The points $z=\frac{-3-\sqrt{5}}{2}$ with $|S'(\frac{-3-\sqrt{5}}{2})| \approx 0.1352 < 1$ and $z=\frac{-3+\sqrt{5}}{2}$ with $|S'(\frac{-3+\sqrt{5}}{2})| \approx 0.0274 < 1$ are attracting fixed points (attractors).

On the other hand, the fixed points $z=0$ and $z=\infty$ are superattracting. In this scenario, $z=-2$ is a critical point of multiplicity two. If the initial value $z_0=-2$ is used, the iteration $z_{n+1}=S(z_n)$ defined in (\ref{S1}) converges to infinity.

\subsubsection{Analysis of the case $K=2$}
In this case, the conjugacy map is:
\begin{equation}\label{S2}
	S(z)=-\frac{z^3(z+5)}{2(3z^2-2z-4)}
\end{equation}
The derivative is given by:
\begin{equation}\label{Sprime2}
	S'(z)=-\frac{3z^2(2z-5)(z+2)^2}{2\left(3z^2-2z-4\right)^2}
\end{equation}
From (\ref{Sprime2}), the strange fixed points are $z=1$ with $|S'(1)|=4.5 > 1$, and $z=-2(3\pm \sqrt{7})$ with $|S'(-6-2\sqrt{7})| \approx 2.8311 > 1$ and $|S'(-6+2\sqrt{7})| \approx 6.9467 > 1$. Consequently, all of these are repelling fixed points.

On the other hand, the fixed points $z=0$ and $z=\infty$ are superattracting. The critical points are $z=5/2$ and $z=-2$ (the latter with multiplicity two). If we use $z_0=5/2$ as the initial value, the iteration $z_{n+1}=S(z_n)$ using $S$ from (\ref{S2}) converges to infinity. When $z_0=-2$ is used, the iteration yields the fixed point $z=1$ in a single step.

\subsubsection{Analysis of the case $K=3$}
In this case, the conjugacy map is:
\begin{equation}\label{S3}
	S(z)=\frac{z^3(z+9)}{z^3-18z^2+27}
\end{equation}
and its derivative is:
\begin{equation}\label{Sprime3}
	S'(z)=\frac{z^2(z+3)^2(z^2-42z+81)}{\left(z^3-18z^2+27\right)^2}
\end{equation}
The strange fixed points are $z=\pm 1$, where $|S'(1)| = 6.4 > 1$ and $|S'(-1)| = 7.75 > 1$. Thus, both are repelling fixed points.

On the other hand, the fixed points $z=0$ and $z=\infty$ are superattracting. The critical points from (\ref{Sprime3}) are $z=3(7\pm 2\sqrt{10})$ and $z=-3$, the latter having multiplicity two. If $z_0=3(7+ 2\sqrt{10})$ is taken as the initial value, the sequence generated by the iteration $z_{n+1}=S(z_n)$ in (\ref{S3}) converges slowly to infinity. Conversely, by using $z_0=3(7- 2\sqrt{10})$, the iteration converges to zero with third-order convergence. Finally, if $z_0=-3$ is chosen, the fixed point $z=1$ is obtained in a single step according to the mapping defined in (\ref{S3}).

\subsection{Stability of the Fixed Points}

To determine the stability of the fixed points of the operator \( S(z) \), we evaluate the modulus of its derivative at each fixed point. The first derivative of \( S(z) \) is given by:

\begin{equation}\label{S'General}
	S'(z)=\frac{2z^2(z+K)^2\left[(K-1)(K-2)z^2-2K(K-1)(K+4)z+3K^2(K+3)\right]}{\left[(K-1)(K-2)z^3+6K(1-K)z^2+2K^2(3-K)z+2K^3\right]^2}
\end{equation}

It is evident from \eqref{S'General} that the fixed points \( z = 0 \) and \( z = \infty \) are always superattracting.
The stability of the remaining fixed points depends on the value of the parameter \( K \). In order to classify them, we evaluate the derivative at each strange fixed point resulting in the following two theorems \ref{thm:z1} and \ref{thm:z23}.

\begin{theorem}\label{thm:z1}
	Let \( z_{1} = 1 \) be a fixed point of the rational operator \( S(z) \).
	Then its stability is governed by the parameter \( K \in \mathbb{C} \) according to
	\[
	2 \left| \frac{(K+1)^2}{K+2} \right|.
	\]
	In particular, $z_1$ is superattracting for $K=-1$, and
	\[
	\begin{cases}
		\left| \dfrac{(K+1)^2}{K+2} \right|  <  \frac{1}{2}  & \Rightarrow z_{1} \text{ is attracting}, \\[10pt]
		\left| \dfrac{(K+1)^2}{K+2} \right|  >  \frac{1}{2}  & \Rightarrow z_{1} \text{ is repelling}, \\[10pt]
		\left| \dfrac{(K+1)^2}{K+2} \right|  =  \frac{1}{2}  & \Rightarrow z_{1} \text{ is neutral}.
	\end{cases}
	\]
\end{theorem}

\begin{proof} Stability of the fixed point \(z_1 = 1\) is determined by the modulus criterion: \( |S'(z_1)| \). Substituting \(z_1 = 1\) into \(S'(z)\) gives \[ |S'(1)| = 2 \left| \frac{(K+1)^2}{K+2} \right|. \] For neutral or repelling behavior, we require \( |S'(1)| \ge 1 \) \(\Rightarrow 2 \left| \frac{(K+1)^2}{K+2} \right| \ge 1 \) \(\Rightarrow \left| \frac{(K+1)^2}{K+2} \right| \ge \frac{1}{2}. \) The strict inequality corresponds to repelling, while equality corresponds to neutral. The attracting case is completely analogous: \( |S'(1)| < 1 \) \(\Rightarrow \left| \frac{(K+1)^2}{K+2} \right| < \frac{1}{2}. \) Hence the modulus criterion fully characterizes the stability of \(z_1\), and \(z_1\) is superattracting for \(K=-1\). \end{proof}

\begin{theorem}\label{thm:z23}
	Let
	\[
	z_{2,3} = \frac{(3 \pm \sqrt{2K+3})K}{K-3}
	\]
	be the other fixed points of the rational operator \(S(z)\).
	Then their stability is governed by the parameter \(K \in \mathbb{C}\) according to
	\begin{equation}\label{eq:modulus-z23}
		\left| S'(z_{2,3}) \right| = 2\left| \frac{7K + 11 \mp (2K + 4)\sqrt{3+2K}}{\left(\sqrt{3+2K}\mp1\right)^2}\right| .
	\end{equation}
	In particular, their stability satisfies
	\[
	\begin{cases}
		\left| S'(z_{2,3}) \right| < 1 & \Rightarrow z_{2,3} \text{ is attracting}, \\[2mm]
		\left| S'(z_{2,3}) \right| > 1 & \Rightarrow z_{2,3} \text{ is repelling}, \\[1mm]
		\left| S'(z_{2,3}) \right| = 1 & \Rightarrow z_{2,3} \text{ is neutral}.
	\end{cases}
	\]
\end{theorem}

\begin{proof}
By substituting \( z_{2,3} \) in \(\left| S'(z) \right| \), we obtain (\ref{eq:modulus-z23}), applying the stability conditions gives the desired result.
\end{proof}

The stability analysis of strange fixed points is conducted using auxiliary functions that evaluate whether the magnitude of the derivative meets the attraction criterion. More precisely, we establish the following definitions:
\begin{equation}\label{eq:stability-fn1}
	\begin{aligned}
		S_{1}(K) &= \min \{ |S'(1)|, 1 \}, \\
		S_{2,3}(K) &= \min \{ |S'(z_{2,3})|, 1 \}.
	\end{aligned}
\end{equation}

Consequently, parameter values \( K \) yielding \( S_{j}(K) = 1 \) mark the stability thresholds, whereas those with \( S_{j}(K) < 1 \) identify regions where the fixed point \( z_{j} \) exhibits attracting behavior.

\noindent
\begin{figure}[H]
	\centering
	\begin{minipage}{0.32\textwidth}
		\centering
		\includegraphics[width=\linewidth]{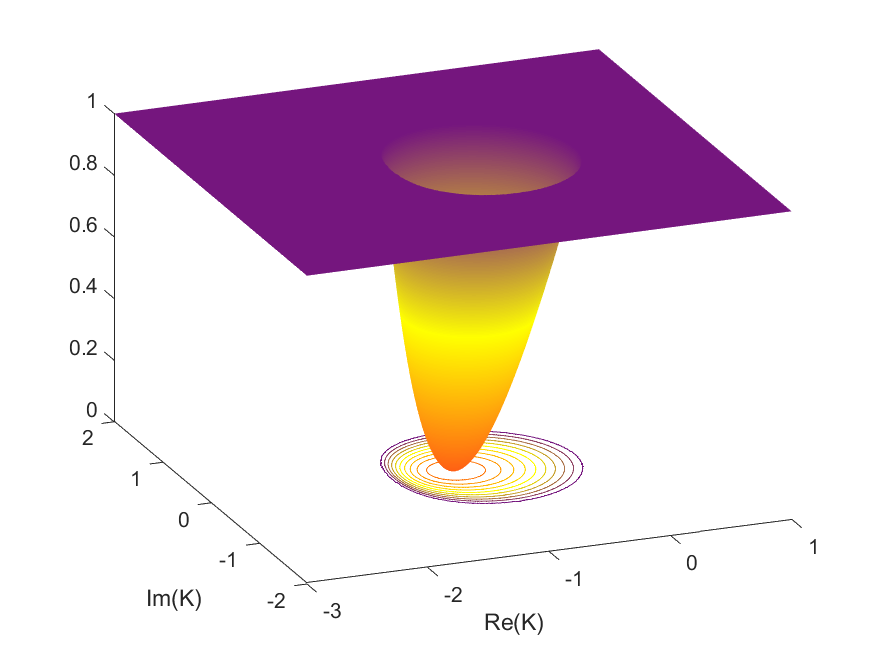}
		\label{fig:image1}
	\end{minipage}
	\hfill
	\begin{minipage}{0.32\textwidth}
		\centering
		\includegraphics[width=\linewidth]{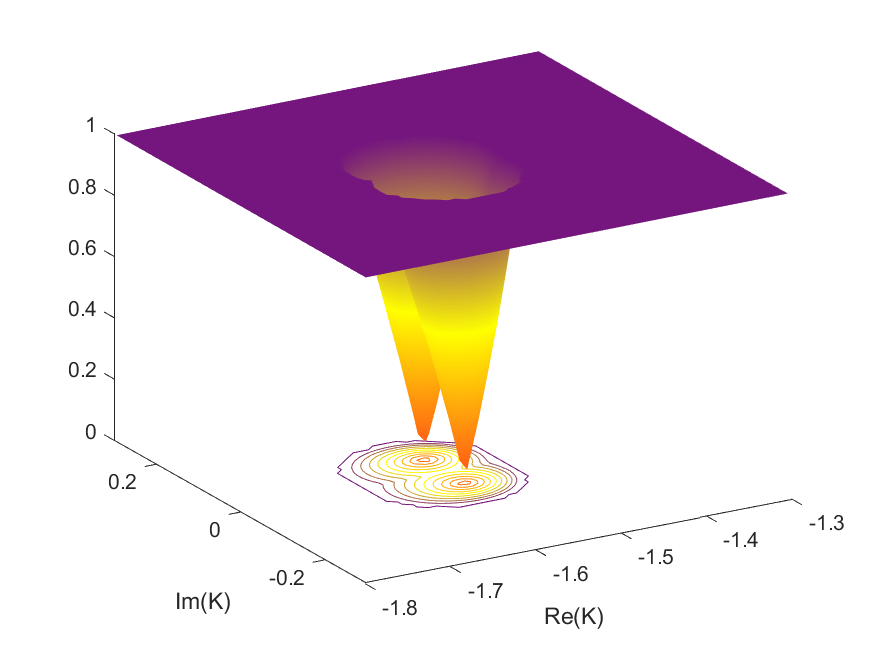}
		\label{fig:image2}
	\end{minipage}
	\hfill
	\begin{minipage}{0.32\textwidth}
		\centering
		\includegraphics[width=\linewidth]{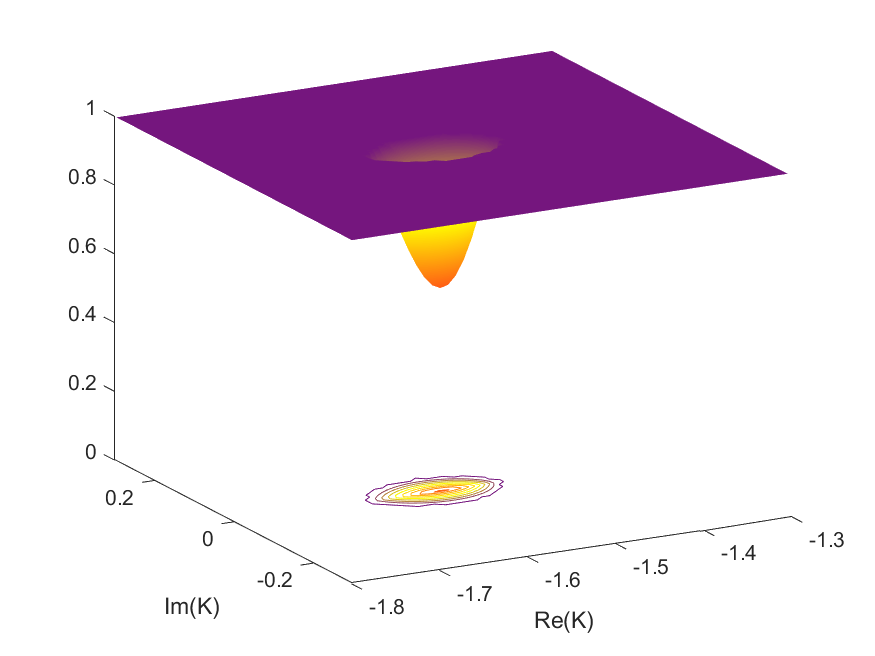}
		\label{fig:image3}
	\end{minipage}
	\caption{Stability functions: left panel corresponds to $z_1$, middle panel to $z_{2}$, and right panel to $z_{3}$}
	\label{fig:figure_1}
\end{figure}

Figure~\ref{fig:figure_1} displays the plots of the stability functions defined in Eq.~\eqref{eq:stability-fn1}, illustrating the boundaries between attraction and repulsion zones within the parameter domain.

\subsection{Study of the Parameter Space}
The parameter space is constructed by iterating a chosen critical point under the action of \(S(z)\) for a wide range of complex values of the parameter \(K\). Each point in the Parameter space corresponds to a specific value of \(K\), and therefore represents a distinct member of the family. In this study, the Parameter space was computed using a grid of \(2000 \times 2000\) points. For each point, up to 50 iterations were performed, with convergence assessed using a tolerance of \(10^{-2}\). This computational approach allows us to visualize how the dynamics of the critical points depend on the parameter values, and to identify regions of stability and divergence.

The critical points of the map \(S(z)\) are obtained by solving \(S'(z) = 0\). Besides the superattracting fixed points at \(z = 0\) and \(z = \infty\), there exist others critical points that govern the dynamics of the family. Specifically, the critical points are given by:

\begin{eqnarray*}\label{Critical_Points}
	C_1 &=& -K \quad \text{(with multiplicity two),} \\
	C_{2,3} &=& \frac{\left((K-1)(K+4) \pm (K+1)\sqrt{(K-1)(K+2)}\right) K}{(K-1)(K-2)}.
\end{eqnarray*}

The resulting dynamics are visualized in Figure \ref{fig:Figure_2} using a color scheme:

\begin{itemize}
	\item Red indicates convergence of the critical point to \(z=0\).
	\item Green indicates convergence of the critical point to \(z=\infty\).
	\item Black represents points for which the iteration does not converge.
\end{itemize}

\begin{figure}[H]
	\centering
	\includegraphics[width=0.32\textwidth]{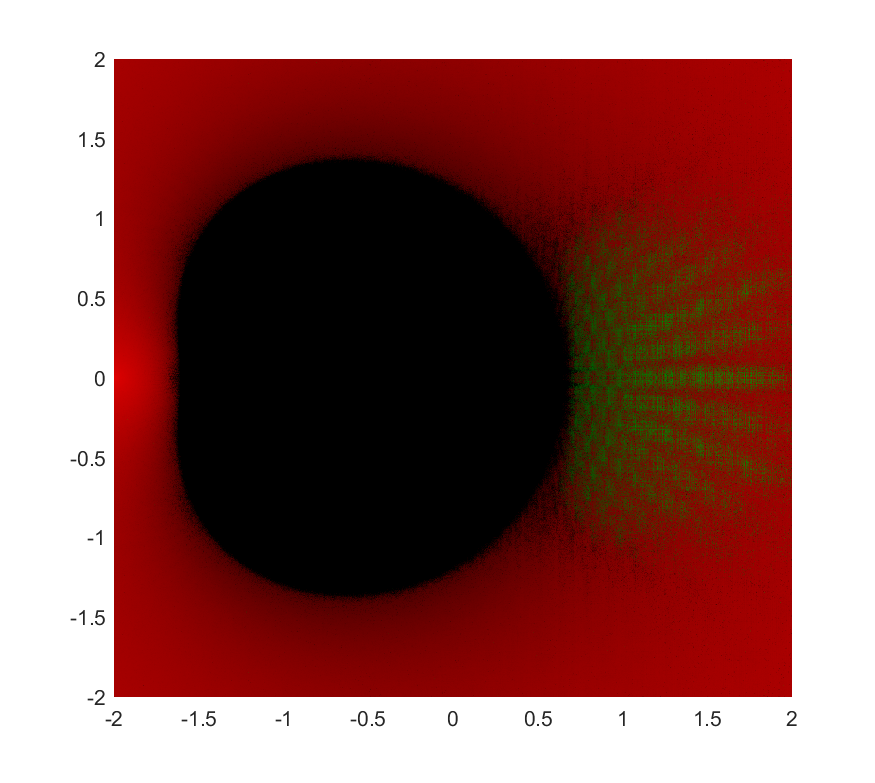}
	\hfill
	\includegraphics[width=0.325\textwidth]{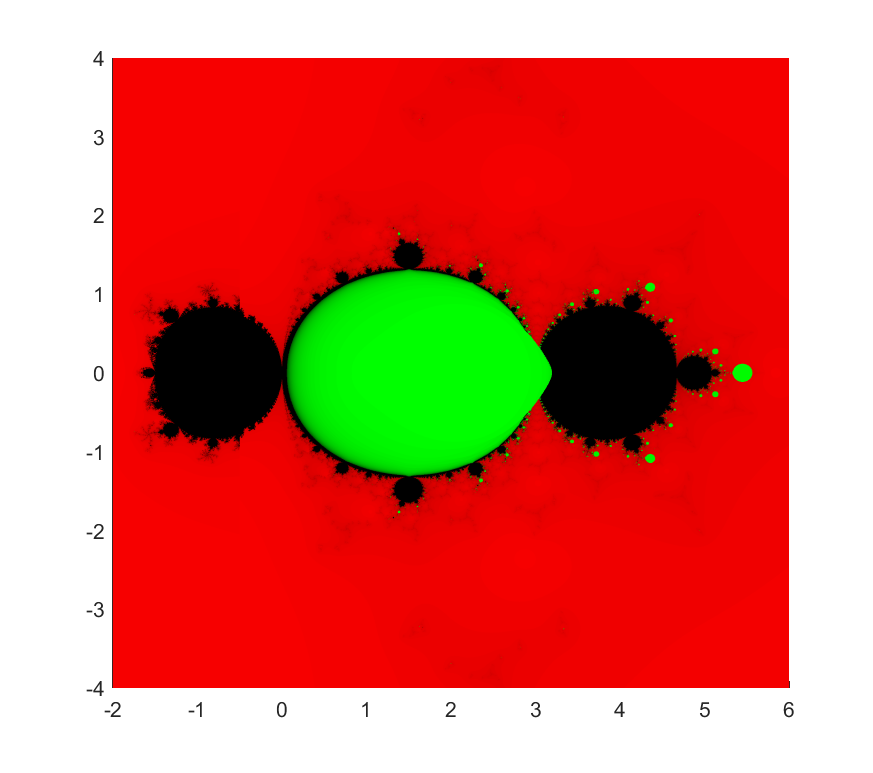}
	\hfill
	\includegraphics[width=0.32\textwidth]{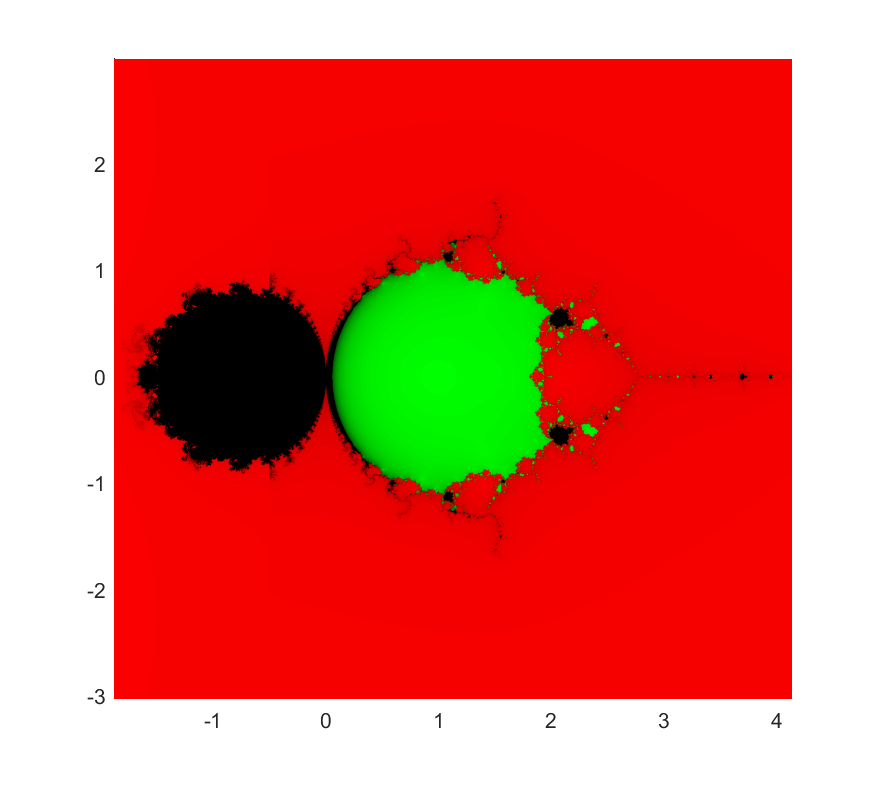}

	\caption{Parameter space associated with the critical points $C_1$, $C_2$, and $C_3$ of $S(z)$.}
	\label{fig:Figure_2}
\end{figure}

\subsection{Dynamical Planes}

This section presents the dynamical planes associated with different values of the Parameter. These plots represent the sets of initial points in the complex plane that, when subjected to an iterative method, converge to one of the roots of the function under study.

A \emph{stable case} refers to parameter values for which there are no strange fixed points, but only zones of convergence to \( 0 \) or \( \infty \), corresponding to the roots \( a \) and \( b \) of the nonlinear equation.
Figure~\ref{fig:Figure_3} illustrates the dynamical planes for some real values of \( K \),  while Figure \ref{fig:Figure_7} depict the dynamical planes for complex values of \( K \).

Each region of the plane is color-coded to indicate its dynamics, where red represents the basin of attraction for the root at $0$, green corresponds to the root at $\infty$, black denotes divergence, and shaded regions represent zones of fastest convergence to the roots. The figures were generated using a $2000 \times 2000$ point grid with $50$ iterations and a convergence tolerance of $10^{-20}$.  These images are composed exclusively of red (convergence to the root $0$) and green (convergence to the root $\infty$) regions, with no black areas present.

\vspace{-1cm}

\begin{figure}[H]
	\centering
	\begin{minipage}{\linewidth}
		\centering
		\includegraphics[scale=0.8]{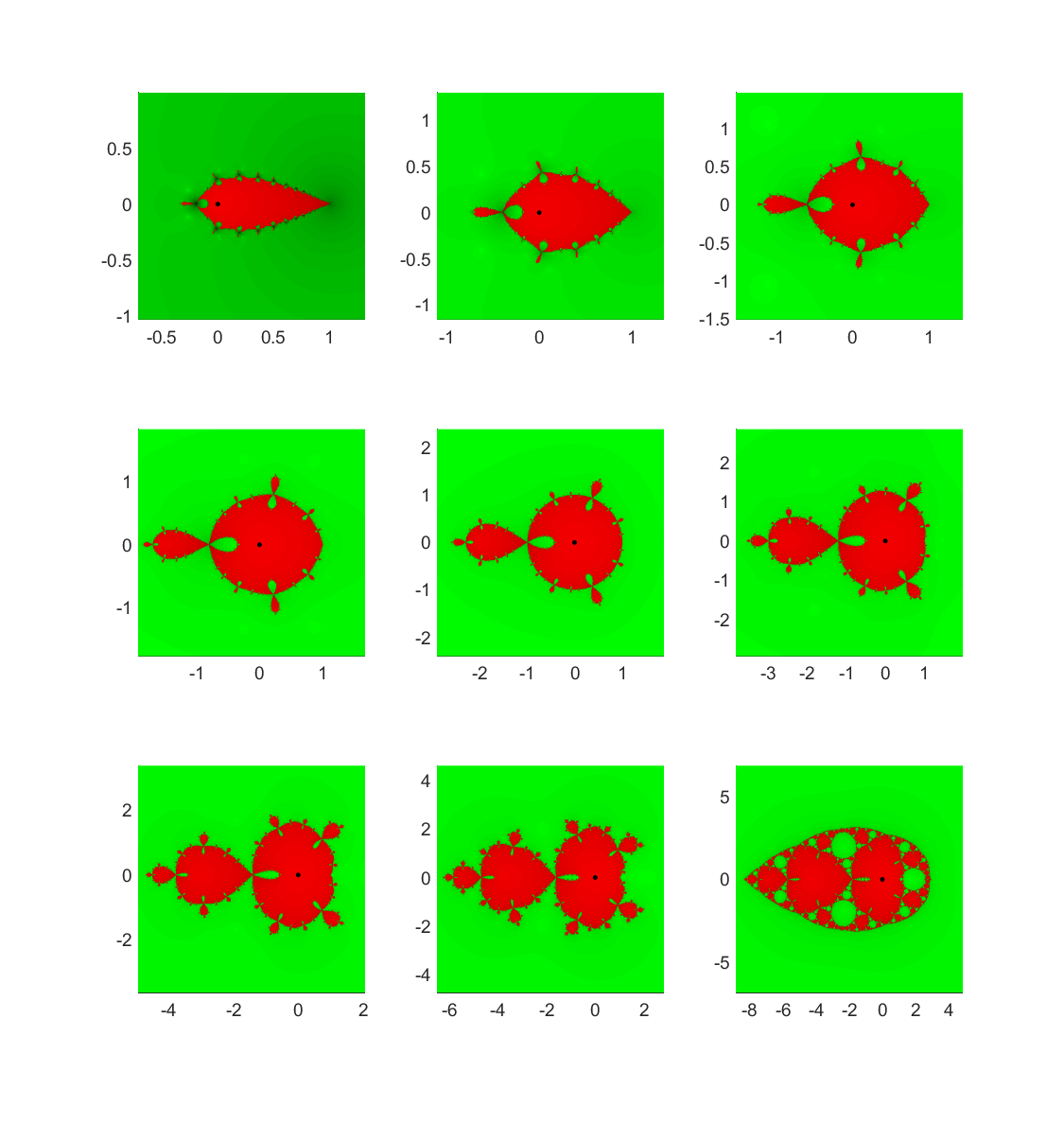}
		\vspace{-40pt}
		\caption{Dynamical planes. $K=0.2:0.2:1.8$}
		\label{fig:Figure_3}
	\end{minipage}
\end{figure}

 This complete absence of divergence confirms that the method is exceptionally stable for these real parameters, successfully converging to one of the two roots for every initial point in the complex domain considered.

\begin{figure}[H]
	\centering
	\begin{minipage}{\textwidth}
		\centering
		\includegraphics[scale=1]{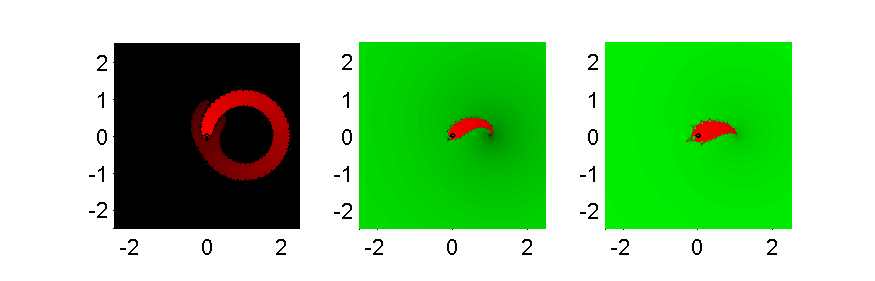}
		\caption{Dynamical planes. $K=0.1I:0.1:0.2+0.1I$}
		\label{fig:Figure_7}
	\end{minipage}
\end{figure}

In Figure \ref{fig:Figure_7}, except for $K=0.1I$, these planes visually validate the theoretical stability of the method for these parameter ranges, showing a healthy, well-partitioned phase space where convergence is guaranteed and relatively fast (as indicated by the shaded intensity regions), making the method highly reliable for root-finding in this domain.

To illustrate unstable cases (Figure~\ref{fig:Figure_8}), some parameter values have been chosen for which the basins of attraction display the presence of strange fixed points.
These parameters lie inside the black colored regions in the parameter space (Figure~\ref{fig:Figure_2}).

\begin{figure}[H]
	\centering
	\begin{minipage}{\textwidth}
		\centering
		\includegraphics[scale=0.65]{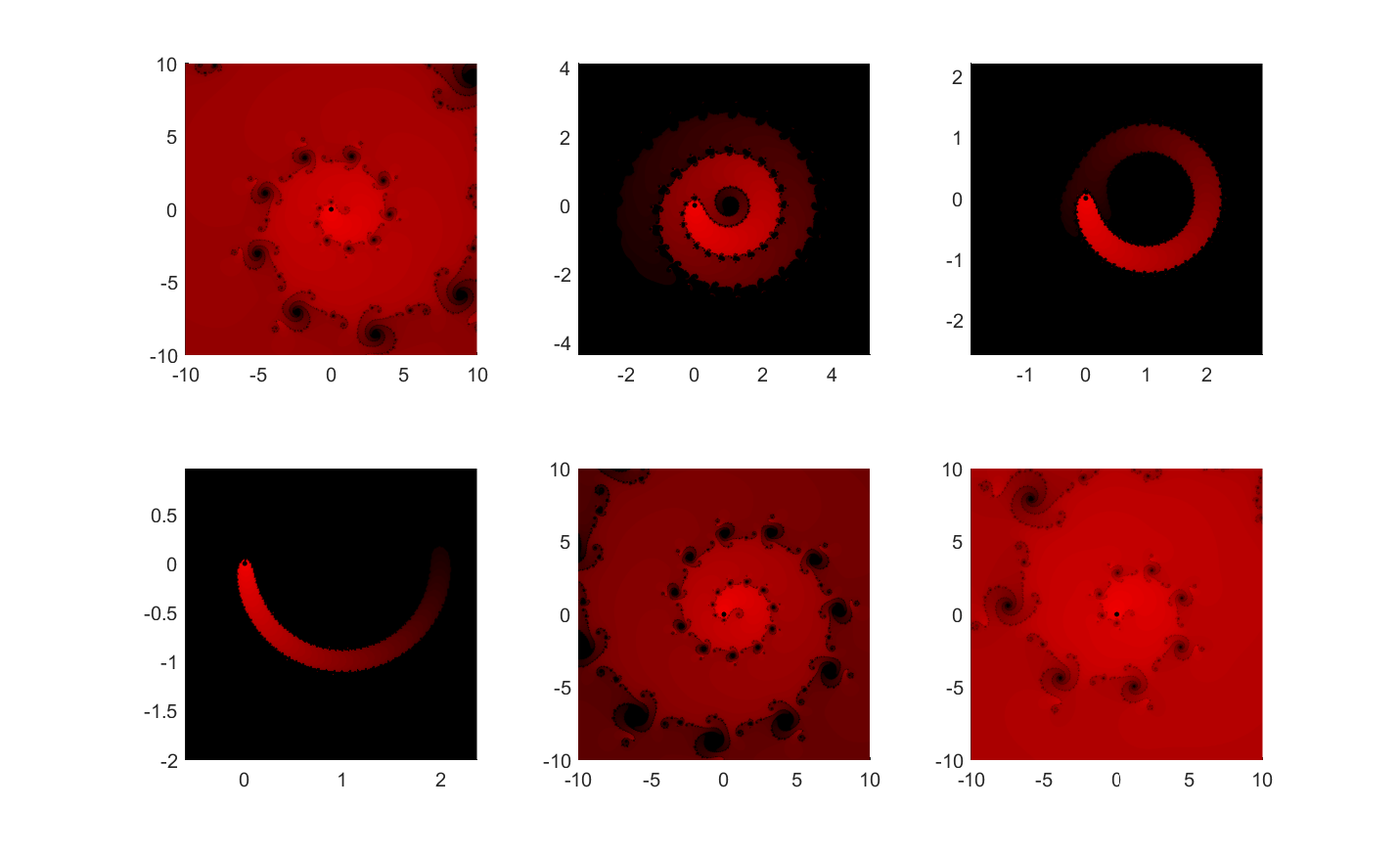}
		\caption{Dynamical planes.  $K= \left\lbrace  -\frac{1}{2}i, -\frac{1}{5}i, -\frac{1}{10}i, -\frac{1}{20}i, \frac{2}{5}i, \frac{3}{5}i \right\rbrace $}
		\label{fig:Figure_8}
	\end{minipage}
\end{figure}

The dynamical planes in Figure~\ref{fig:Figure_8} present a markedly different and unstable behavior. Here, the regions are composed solely of black (divergence) and red (convergence to \( z_0 = 0 \)) areas, with a complete absence of green regions corresponding to convergence to the other root. For the purely imaginary values \( K = -\frac{1}{5}i, -\frac{1}{10}i, -\frac{1}{20}i \), the plots are dominated by extensive black areas, indicating widespread divergence of the iterative method. Within these large divergence zones, only small, isolated red structures appear. The presence of these large black regions demonstrates a significant loss of stability and reliability in the numerical scheme when \( K \) takes these particular imaginary values.

\subsection{Basins of attraction}
To complete this work, the dynamical analysis of the function
\[
g(x) = (x-1)^m (x+1)^n
\]
is carried out by computing its fixed and critical points. Subsequently, the basins of attraction of the roots \( z = 1 \) and \( z = -1 \) are graphically represented.

Applying the modified Chebyshev method to \( g(x) \) yields the following iteration map.

Let us define the polynomials \( u_i(K) \) corresponding to the coefficients of the numerator in the expression involving \( x \):

\[
\begin{aligned}
	u_1(K) &= K^2 + 3K + 2, \\
	u_2(K) &= 2K^3 + 4K^2 - 6, \\
	u_3(K) &= 6K^3 + 6K^2 - 6K + 6, \\
	u_4(K) &= 6K^3 - 4K^2 - 2, \\
	u_5(K) &= 2K^3 - 7K^2 + 3K .
\end{aligned}
\]

Then the iteration function can be written compactly as:

\[
x_{n+1}=\frac{u_1(K) \, x_{n}^4 + u_2(K) \, x_{n}^3 + u_3(K) \, x_{n}^2 + u_4(K) \, x_{n} + u_5(K)}
{2\big( (K+1)x_{n} + (K-1) \big)^3 } .
\]

So, the iteration function in the complex plane is given by

\[
G(z)=\frac{u_1(K) \, z^4 + u_2(K) \, z^3 + u_3(K) \, z^2 + u_4(K) \, z + u_5(K)}
{2\big( (K+1)z + (K-1) \big)^3 } .
\]

Its fixed points are:

\[
\begin{aligned}
	z_{0} &= 1, \\[4pt]
	z_{1} &= -1, \\[-4pt]
	z_{2,3} &= -\frac{2K^{2} + K\bigl(1 \pm 2\sqrt{2K+3}\bigr) - 3}
	{(2K+3)(K+1)}
\end{aligned}
\]

where \( z_0 = 1 \) corresponds to one of the roots of the polynomial, while \( z_1, z_2, \) and \( z_3 \) are strange fixed points.

\subsubsection{Stability of the fixed points}

It is clear that $z_0 = 1$ is a superatractor fixed point, since it is the root of $g$ with multiplicity $m$. It is necessary to study the stability of the fixed points $z_1 = -1$ and $z_{2,3}$. To do this, we differentiate $G$:

\begin{equation}
	G'(z) = \frac{(z -1)^{2} [(z^{2}+6z+5) K^{3}+4 (z+2)^{2} K^{2}+(5z^{2}+6z-11) K +2 (z-1)^{2}]}{2 [(z+1) K +z-1]^{4}}
\end{equation}

and evaluate these values. The operator $G'(z)$ in $z = -1$ gives:

\begin{equation}
	|G'(-1)| = \left| \frac{1}{2} K^{2}-\frac{3}{2} K +1 \right|
\end{equation}

\vspace{1em}
In the following result we present the stability of the fixed point $z = -1$ in theorem \ref{thm:especial1}, and for $z_{2,3}$ in theorem \ref{thm:especial2}.

\begin{theorem}\label{thm:especial1}
	Let \( z_{1} = -1 \) be a fixed point of the rational operator \( G(z) \).
	Then its stability is governed by the parameter \( K \in \mathbb{C} \) according to
	\[
	|G'(-1)| = \left| \frac{1}{2} K^{2} - \frac{3}{2} K + 1 \right|.
	\]
	In particular, $z_1$ is superattracting when $G'(-1) = 0$, and
	\[
	\begin{cases}
		\left| \left(K -1\right) \left(K -2\right) \right|  < 2  & \Rightarrow z_{1} \text{ is attracting}, \\[6pt]
		\left| \left(K -1\right) \left(K -2\right) \right|  > 2  & \Rightarrow z_{1} \text{ is repelling}, \\[6pt]
		\left| \left(K -1\right) \left(K -2\right) \right|  = 2  & \Rightarrow z_{1} \text{ is neutral}.
	\end{cases}
	\]
\end{theorem}

\begin{proof}
	The stability of the fixed point \( z_1 = -1 \) is determined by the value of \( |G'(-1)| \).
	Direct computation gives
	\[
	G'(-1) = \frac{1}{2}K^2 - \frac{3}{2}K + 1,
	\]
	so that
	\[
	|G'(-1)| = \left| \frac{1}{2}K^2 - \frac{3}{2}K + 1 \right|.
	\]

	Factorizing the quadratic,
	\[
	\frac{1}{2}K^2 - \frac{3}{2}K + 1 = \frac{1}{2}\big( K^2 - 3K + 2 \big) = \frac{1}{2}(K-1)(K-2).
	\]
	Thus,
	\[
	|G'(-1)| = \frac{1}{2}\left| (K-1)(K-2) \right|.
	\]

	For attracting behavior we require \( |G'(-1)| < 1 \), which is equivalent to
	\[
	\frac{1}{2}\left| (K-1)(K-2) \right| < 1 \quad \Longrightarrow \quad \left| (K-1)(K-2) \right| < 2.
	\]
	For repelling behavior, \( |G'(-1)| > 1 \) yields
	\[
	\left| (K-1)(K-2) \right| > 2.
	\]
	The neutral case corresponds to equality:
	\[
	\left| (K-1)(K-2) \right| = 2.
	\]

	Finally, \( z_1 = -1 \) is superattracting when \( G'(-1) = 0 \), which occurs for \( K = 1 \) or \( K = 2 \).
\end{proof}

\begin{theorem}\label{thm:especial2}
	Let
	\[
	z_{2,3} = -\frac{2K^{2} + K\bigl(1 \pm 2\sqrt{2K+3}\bigr) - 3}
	{(2K+3)(K+1)}
	\]
	be the strange fixed points of the rational operator \(G(z)\).
	Then their stability is governed by the parameter \(K \in \mathbb{C}\) according to

	\begin{equation}\label{eq:modulus-z24}
		\left| G'(z_{2,3}) \right| = \left| \frac{3 K^{2}+11 K \pm \left(K^{2}+K -3\right) \sqrt{2 K +3}+10}{\left(K +1\right)^{2}} \right| .
	\end{equation}

	In particular,
	\[
	\begin{cases}
		\left| G'(z_{2,3}) \right| < 1 & \Rightarrow z_{2,3} \text{ are attracting}, \\[2mm]
		\left| G'(z_{2,3}) \right| > 1 & \Rightarrow z_{2,3} \text{ are repelling}, \\[1mm]
		\left| G'(z_{2,3}) \right| = 1 & \Rightarrow z_{2,3} \text{ are neutral}.\\[1mm]
		\left| G'(z_{2,3}) \right| =0 & \Rightarrow z_{2,3} \text{ are superattracting}.
	\end{cases}
	\]
\end{theorem}

\begin{proof}
	Direct substitution of \( z_{2,3} \) into the derivative \( G'(z) \) yields the expression in (\ref{eq:modulus-z24}).
	Applying the standard stability criteria for fixed points of rational maps gives the stated conditions.
\end{proof}

The dynamic planes of the family in study will be presented, to find multiple roots of polynomials of the form $f(z) = (z-a)^{nk}(z-b)^{n}$, where $k$ is a proportionality constant for the exponent (multiplicity) of one of the roots of this polynomial.
The dynamic planes were generated using a mesh of $1000 \times 1000$ points in the complex region $[-5, 5] \times [-5, 5]$. For each point in the grid, the iterative method was applied with a maximum of $30$ iterations and a convergence tolerance of $1 \times 10^{-5}$. As before, red regions represents the convergence to the root $z_0 = 1$, green regions to the root $z_1=-1$, and black regions for divergence.

\begin{figure}[H]
	\centering
	\begin{minipage}{\textwidth}
		\centering
		\includegraphics[scale=0.5]{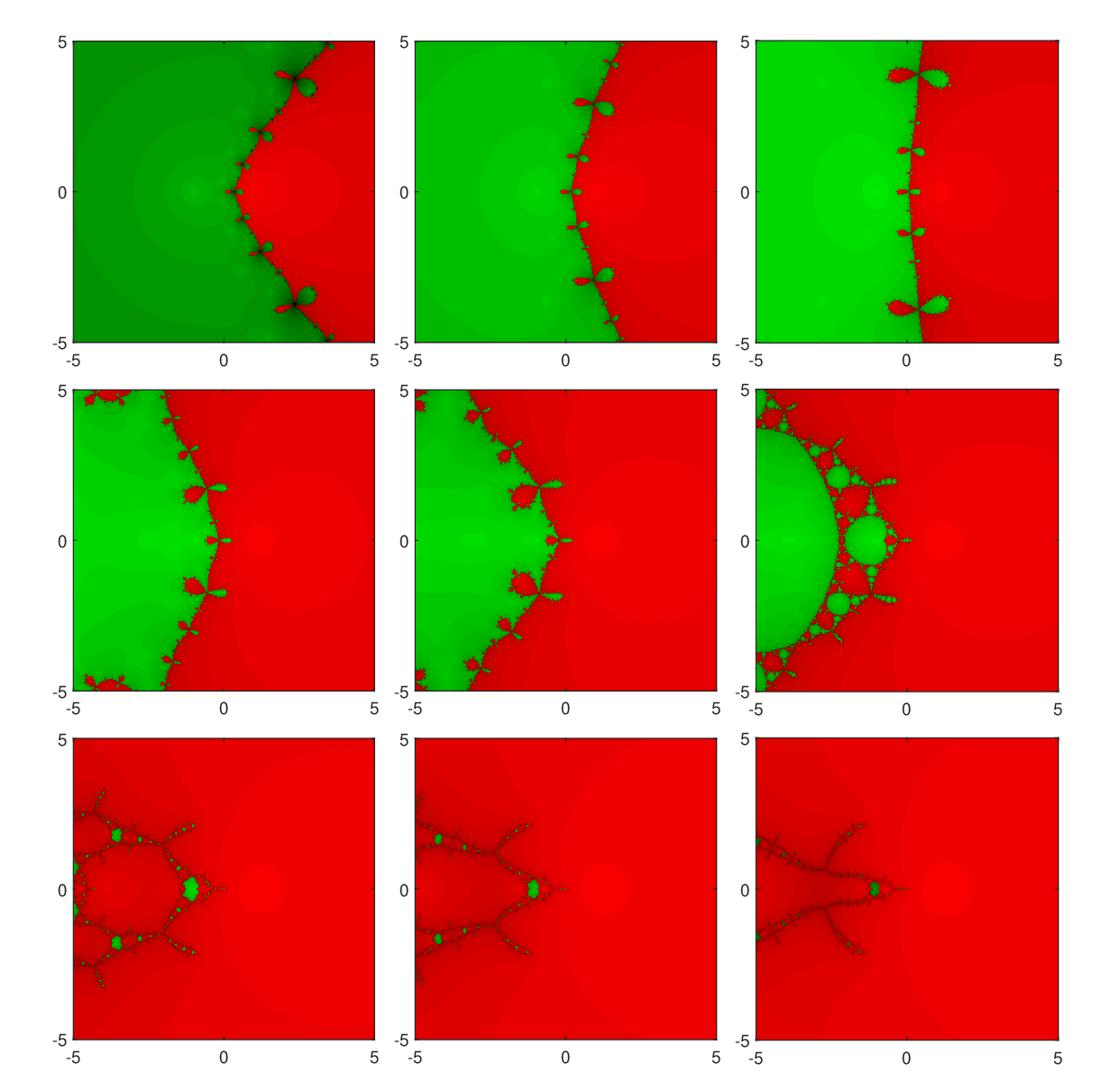}
		\caption{Dynamical planes. $ g(x)=(x-1)^m (x+1)^n; m=Kn; K = \left\{ \frac{1}{2}, \frac{7}{10}, \frac{9}{10}, \frac{7}{5}, \frac{8}{5}, \frac{9}{5}, \frac{23}{10}, \frac{5}{2}, \frac{27}{10} \right\}$}
		\label{fig:Aplication1}
	\end{minipage}
\end{figure}

Following the previous analysis, the figure displays only red regions (convergence to \( z_0 = 1 \)) and green regions (convergence to \( z_1 = -1 \)), with an almost complete absence of black regions (divergence). This indicates that the iterative method used is remarkably stable and efficient for this specific family of polynomials within the studied domain.

The clear separation between the red and green regions suggests that the basins of attraction for each root are well-defined and that the method consistently converges to one root or the other depending on the initial condition. The structure of the regions also reveals certain symmetric or fractal patterns, which is typical in dynamical planes associated with iterative methods in the complex plane.

\section{Result and discussion}

This study examines the complex dynamics of the modified Chebyshev method when applied to a specific class of polynomials characterized by two roots with multiplicities defined as $m = Kn$. By establishing the scaling theorem and conjugation mapping, the research effectively reduces the family of polynomials to a simplified rational operator. 

The analytical determination of fixed and critical points then provides a foundation for understanding the stability and convergence properties of the method under these specific conditions.The investigation further explores the parameter space, utilizing numerical visualizations to map out the corresponding dynamical planes. 

By analyzing the basins of attraction for selected examples, the study highlights the stability of the fixed points and identifies potential regions of divergence or chaotic behavior. These results demonstrate the robustness of the modified Chebyshev method in handling multiple roots.

\bibliographystyle{unsrt}

\end{document}